\documentclass[11pt]{article}
\usepackage[utf8]{inputenc}
\usepackage{GMathFinalV2, lineno, setspace}

\title{The Chow Ring Classes of $\PGL_3$ Orbit Closures in $\GG(1,5)$}
\author{Gaurav (Dhruv) Goel\footnote{Herchel Smith Harvard Undergraduate Summer Research Program 2022}}

\date{October 2022}

\begin{document}

\maketitle

\abstract{The space of all pencils of conics in the plane $\PP V$ (where $\dim V = 3$) is a projective Grassmannian $\GG(1,\PP\Sym^2 V^*)$  and admits a natural $\PGL(V)$ action. It is a classical theorem that this action has exactly eight orbits, and in fact that the orbit of a pencil $\ell\subset \PP\Sym^2 V^*$ is determined completely by its position with respect to the Veronese surface $X\subset \PP\Sym^2V^*$ of rank 1 conics and its secant variety $S(X)\subset \PP\Sym^2 V^*$, which is the cubic fourfold of rank 2 conics. In this paper, we present some geometric descriptions of these orbits. Then, using a mixture of direct enumerative techniques and some Chern class computations, we present a calculation of the classes of the orbit closures in the Chow ring of this Grassmannian (and consequently also of their degrees under the Plücker embedding $\GG(1,\PP\Sym^2 V^*)\injto \PP\Lambda^2 \Sym^2 V^*$).}

\newpage 
\tableofcontents

\newpage
\section{Introduction}

Let $V$ be an $n+1$ dimensional vector space. The complete linear system \[\PP H^0\OO_{\PP V}(2)\cong \PP\Sym^2 V^*\cong \PP^{n(n+3)/2}\] of quadrics on $\PP V=\PP^n$ admits a natural action by $\PGL(V)=\PGL_{n+1}$. It is a standard result from linear algebra that this $\PGL(V)$ action has exactly $n+1$ orbits, classified by the rank of the corresponding bilinear form (at least when the base field is algebraically closed). As natural generalization of this question, we can ask the same question about \textit{families} of quadrics in $\PP V$. The space of all $d$-dimensional linear systems of quadrics on $\PP V$ is the projective Grassmanian $\GG(d, \PP\Sym^2 V^*)$, which consequently also carries a natural $\PGL(V)$ action (the above being the case $d=0$). The orbits of this $\PGL(V)$ action are irreducible locally closed subsets of $\GG(d,\PP\Sym^2 V^*)$, and we can ask many questions about them, e.g.: What are their dimensions? Are there only finitely many? If so, what is the corresponding orbit class space or its Hasse diagram (see \cite{Hasse})? Where is each orbit closure smooth or singular? What are the classes of the orbit closures in the Chow ring $A\GG(d,\PP\Sym^2 V^*)?$

When $n=1$, we're looking at the $\PP^2$ of quadratic polynomials on $\PP^1$, and we have $d\in\{0,1,2\}$. For $d=0$, we have two orbits: the dense orbit (corresponding to two distinct points) and the discriminant conic (corresponding to double points); the $\PGL(V)=\PGL_2$ in this case is exactly the subgroup of $\Aut \PP^2=\PGL_3$ preserving the discriminant conic. For $d=1,$ the situation is  entirely dual: we have the dense orbit (of lines transverse to the discriminant conic) and the dual to the discriminant conic. Obviously, the case $d=2$ is trivial.

For $n\geq 2$, the simplest nontrivial case is $d=1$, corresponding to pencils of quadrics. This corresponds to a $\PGL_{n+1}$ action on $\GG(1, n(n+3)/2)$. In this case, the classification is nontrivial but well understood, at least in the case of pencils that contain at least one smooth quadric. Here there are two invariants: a discrete invariant called the \textit{Segre symbol} and a continuous invariant corresponding to the multiset of $n+1$ points in $\PP^1$ defined by the vanishing of the discriminant of the pencil. It is a classical theorem of Weierstrass and Segre that these are the \textit{only} invariants (for a proof, see this historical account \cite{FMS} or the classic textbook \cite{HodgePedoe}). The result of Weierstrass-Segre can also be interpreted very geometrically as being a classification obtained by considering the positions of the lines with respect to the various $\PGL_{n+1}$ orbit closures in $\PP^{n(n+3)/2}$ and the types of singularities of the base locus of the pencil (see \cite{Dimca}).

Since multisets of $4$ or more points have continuous invariants\footnote{For instance, for $4$ distinct points on $\PP^1$, we have the cross ratio $\lambda$, or more precisely the corresponding $j$-invariant \[j(\lambda) = 256\frac{(1-\lambda+\lambda^2)^3}{\lambda^2(1-\lambda)^2}.\]}, the largest $n$ for which we can expect to have finitely many orbits is $n=2$, i.e. the case of planar conics. The classification of pencils of plane conics (i.e. the case $n=2, d=1$) was first obtained over $\RR$ and $\CC$ by Jordan in 1906 (see \cite{Jordan1}, \cite{Jordan2}, or see \cite[Ch. 16]{Berger} if you prefer English)\footnote{By contrast, the case $n=d=2$ (i.e. nets of plane conics) was investigated only recently in 2021 by Abdallah, Emsalem, and Iarobbino in \cite{AEI}. Here there is no dense orbit, but rather a 1-parameter family of dimension 8 orbits in $\GG(2,5)$, along with 14 other orbits. This family is parametrized by the $j$-invariant of the smooth plane cubic that is the locus of singular conics in that net.}. There are exactly 8 orbits in $\GG(1,5)$, which we call $\OO_{i}$ for $i=1,\dots, 8$. There are five orbits, $\OO_1, \OO_2,\OO_3,\OO_4,$ and $\OO_5$, of pencils containing at least one smooth conic, corresponding to the general pencils, simply tangent pencils, bitangent pencils, osculating pencils, and superosculating pencils respectively (c.f. \cite[\S 16.5]{Berger} types I-V) and three types $\OO_6, \OO_7, \OO_8$ of singular pencils (i.e. pencils containing only singular conics) described in detail in \S 2.2. The Hasse diagram of their inclusions (i.e. a poset of the $\OO_i$ with an arrow $\OO_i\to \OO_j$ iff $\OO_j$ is open in $\ol \OO_i\minus \OO_i$) is given in Figure \ref{Hasse}.

\begin{figure}[h]
    \centering
    \begin{tikzcd}
    &&\OO_3 \arrow{dr}&&\OO_6\arrow{dr}\\
    \OO_1\arrow{r}&\OO_2\arrow{ur}\arrow{dr}&&\OO_5\arrow{ur}&&\OO_8\\
    &&\OO_4\arrow{ur}\arrow{rr}&&\OO_7\arrow{ur}
    \end{tikzcd}
    \caption{Hasse Diagram of the $\PGL_3$ action on $\GG(1,5)$.}
    \label{Hasse}
\end{figure}
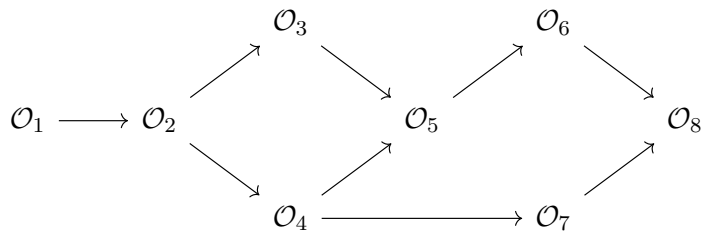

Each of these orbits has very interesting algebraic geometry associated to it. For instance:
\begin{enumerate}
    \item The ideals of $\ol \OO_5, \ol \OO_6, \ol \OO_7$ are generated by quadratic polynomials in the Plücker coordinates.
    \item The orbit $\OO_2$ is birational to a $\PP^3$ bundle over a cubic fourfold, namely the smooth locus $S_\sm$ of the secant variety $S$ to the Veronese surface.
    \item The orbit $\OO_3$ is birational to a $\PP^4$ bundle over $\PP^2$.
    \item The orbit $\OO_5$ is birational to bundle over $\PP^2$ with fibers quadric threefolds.
    \item The closure $\ol \OO_5$ is non-normal and singular along $\ol \OO_6$, which is isomorphic to the Hilbert scheme $\Hilb^2 \PP V^*$ of pairs of lines in the plane.

    \item The closure $\ol \OO_7$ is the Segre fourfold $\PP^2\times \PP^{2*}$ and $ \OO_8$ inside it is the universal line, and so $\OO_8$ is a $\PP^1$ bundle over $\PP^2$.
    \item The closures $\ol \OO_6$ and $\ol \OO_7$ are the irreducible components of the Fano variety of the cubic fourfold $S$.
\end{enumerate}
See \cite{HMS} or\cite[\S 2]{Lyu} for proofs of some of the facts above; most of these will also be reproven in the course of our discussion, see \S 3. In this paper, we compute the classes of the closures $\ol \OO_i$ in the Chow ring $A\GG(1,5)$ and their degrees under the Plücker embedding $\GG(1,5)\injto \PP^{14}$.\footnote{The Plücker degrees of the Schubert cycles are well-known and can be computed by the hook length formula (see \cite{Smirnov} Theorem 2.3.1). Since the Plücker degree of any other closed subscheme of the Grassmannian is the linear combination of its Chow ring coefficients obtained by plugging in the appropriate Plücker degrees of Schubert cycles, the Chow ring class is a more refined quantity associated to the orbits than its Plücker degree.
Some of these degrees can be found in \cite{FMS} Example 5.2 and \cite{HMS} Proposition 1.1, and it is a good sanity check that our methods, although very different from theirs, give the same results.} Our main result is:

\begin{theorem}
The classes in the Chow ring $A\GG(1,5)$ of the closures of the orbits $\OO_1,\dots, \OO_8$ described above, and their degrees under the Plücker embedding are summarized in Table \ref{results}.

\begin{table}[h]
\label{results}
    \centering
    \begin{tabular}{|c|c|c|c|c|}
    \hline Orbit & Base Locus Type &Codimension &$ A\GG(1,5)$ Class of Closure &Plücker Degree
    \\
    \hline
    $\OO_1$& $(1,1,1,1)$&0 &$\sigma_0$&14
    \\
    $\OO_2$& $(2,1,1)$ &1& $6\sigma_1$&84\\
    $\OO_3$&$(2,2)$ &2& $4\sigma_2$&36\\
    $\OO_4$&$(3,1)$&2&$6\sigma_2+9\sigma_{1,1}$&99\\
    $\OO_5$&$(4)$&3&$4\sigma_3+8\sigma_{2,1}$&56\\
    \hline
    $\OO_6$&$\{*\}$&4&$3\sigma_{3,1}+6\sigma_{2,2}$&21\\
    $\OO_7$&$L\cup\{*\}:*\notin L$&4&$6\sigma_{3,1}+3\sigma_{2,2}$&24\\
    $\OO_8$&$L\cup\{*\}:*\in L$&5&$6\sigma_{4,1}+6\sigma_{3,2}$&18\\\hline
    \end{tabular}
    \caption{Classes in $A\GG(1,5)$ of the $\PGL_3$ orbit closures $\ol \OO_i$ for $i=1,\dots, 8.$}
\end{table}
\end{theorem}
\vspace{8pt}
The rest of this paper is organized as follows. In \S 2, we establish our notation and conventions, and then recall the classification of these orbits and show how this relates to the positions of these lines with respect to the Veronese surface and its secant variety. In \S 3, we give direct enumerative arguments to compute the Chow ring classes of these orbits. In \S 4, we give a second proof of most of these results by computing some Chern classes. In \S 5, we explain how to modify our arguments in the case of positive characteristic $p>2$. Finally, in \S 6, we present several suggestions for future work in this direction.

The author was supported by the Herchel Smith Harvard Undegraduate Research Program Summer 2020, and would like to thank Professor Joe Harris for his supervision and guidance.
\newpage 
\section{Set-Up}
\subsection{Notation and Conventions}

We work over a fixed algebraically closed field $K$ of characteristic 0.\footnote{For all practical purposes, $K=\CC$ (this is somewhat justified by the Lefschetz principle). Most of our results can be extended to algebraically closed fields of positive characteristic $p>2$ without much difficulty, although proofs of results invoking Kleiman's theorem would have to be replaced by a careful analysis of generic transversality of the corresponding cycles (see \S 5). The situation in characteristic 2 is genuinely quite wonky. For  starters, we cannot use the correspondence between plane conics and symmetric bilinear forms. Further, there are other pathologies, e.g. any smooth plane conic is a \textit{strange curve}, i.e. there is a unique point, called the \textit{strange point}, common to all tangents to the conic. We will not pursue characteristic 2 any further; see \cite{Vainsencher} for more on conics in characteristic 2, and see the work of Bhosle, e.g. \cite{Bhosle}, for more on pencils of quadrics in characteristic 2.} Let $V$ be a three dimensional $K$-vector space and let $\PP^2 = \PP V$ be the corresponding projective plane with coordinates $x,y,z$. By $\PP^5$, we always denote the space $\PP\Sym^2 V^*$ of conics on $\PP^2$, and give it coordinates $a,b,c,h,e,f$ (in that order) so that the universal conic \[\Xi=\{(C,p):p\in C\}\subset\PP^5\times \PP^2\]
is cut out by the equation \[ax^2+2hxy+by^2+2exz+2fyz+cz^2=0.\]
In other words, the symmetric bilinear form on $V$ corresponding to a conic with coordinates $a,b,c,h,e,f$ is defined by the matrix \[\begin{bmatrix}
a&h&e\\h&b&f\\e&f&c
\end{bmatrix}.\]
Let \[\Delta:=\det \begin{bmatrix}
a&h&e\\h&b&f\\e&f&c
\end{bmatrix}=abc+2hef - af^2-be^2-ch^2\in K[a,b,c,h,e,f]\]
be the determinant of this matrix and let $S=\VV(\Delta)\subset \PP^5$ be the cubic fourfold
of singular conics. Since the partial derivatives of $\Delta$ are exactly the six distinct $2\times 2$ minors of this matrix (upto scalars), it follows that the singular locus $X=S_{\sing}$ is exactly the locus of rank 1 conics. It is easy to see that $X$ is the Veronese surface (the image of the ``squaring map'' $\PP V^*\to \PP\Sym^2 V^*$), and that $S$ is its secant (and tangent) variety $S=S(X)=\tan(X)$, i.e. the union of all secant lines to $X$, which also happens to be the union of all tangent lines to $X$ (c.f. \cite[Ch. 11]{Harris}).

We let $\GG(1,5)$ denote the Grassmannian $\GG(1,\PP\Sym^2 V^*)$ of pencils of conics. For each $r\geq 0$, we denote by $\Phi_r$ the incidence correspondence \[\Phi_r:=\{(\ell,p):m_p(\ell, S)\geq r\}\subseteq \GG(1,5)\times \PP^5.\]
We let $\pi_1$ and $\pi_2$ denote the projections on $\GG(1,5)$ and $\PP^5$ respectively:

\begin{center}
    \begin{tikzcd}[column sep = small]
    &\Phi_r\arrow[swap]{dl}{\pi_1}\arrow{dr}{\pi_2}\\\GG(1,5)&&\PP^5
    \end{tikzcd}
\end{center}
Each $\Phi_r$ is clearly a closed subset; it is the vanishing locus of a certain global section $\tau_{\Delta}$ of a vector bundle $\SE^r$ on $\Phi$ (see \S \ref{section: another_computation}).  For $r=0$, this is all of $\GG(1,5)\times \PP^5$. For $r\geq 1$, this is a nonempty proper closed subset. Since $\deg S=3$, for $r\geq 4$ these loci all coincide and are simply \[\Phi_4:=\{(\ell,p):p\in \ell\subseteq S\}\subset \GG(1,5)\times \PP^5.\]

Finally, for each ordered pair $(a,b)$ of integers $4\geq a\geq b\geq 0$ and complete flag $\SV$ in $\PP^5$, we let $\Sigma_{a,b}(\SV)\subset \GG(1,5)$ denote the corresponding Schubert cycle. When $\SV$ is irrelevant, we simply omit it and write $\Sigma_{a,b}$. We let $\sigma_{a,b}=:=[\Sigma_{a,b}]\in A \GG(1,5)$. Finally, we let $\sigma_{a}:=\sigma_{a,b}$ as usual. Finally, we present the following helpful lemma:

\begin{lemma}
\label{lemma: pluecker_degree_lines}
    Let $N\geq 1$ be any integer. Then for any ordered pair $(a,b)$ of integers with $N-1\geq a\geq b\geq 0$, the Plücker degree of the Schubert cycle $\Sigma_{a,b}$ is \[\deg \Sigma_{a,b} = \binom{2N-2-a-b}{N-1-b}\frac{a-b+1}{N-b}.\]
\end{lemma}
\begin{proof}
    By the hook length formula (\cite{Smirnov} Theorem 2.3.1), we have \[\deg \Sigma_{a,b} = \frac{(2N-2-a-b)!}{
    \left(\prod_{k=1}^{N-1-a} (N+1-b-k)\right)\left(\prod_{k=N-a}^{N-1-b}(N-b-k)\right)\left(\prod_{k=1}^{N-1-a} (N-a-k)\right)
    },\]
    which simplifies to the suggested quantity.
\end{proof}
\subsection{Classification of $\PGL_3$-orbits and Hasse Diagram}

In this section, we recall the classification of $\PGL_3$-orbits, interpret them geometrically based on the position of the line $\ell\subset \PP^5$ with respect to $X$ and $S$, and figure out the Hasse diagram of the orbit closures. First recall the following definition:

\begin{definition}
Given a pencil $\ell$ of conics in $\PP^2$, its \textit{base locus} is the closed subscheme \[B_\ell:=\bigcap_{C\in \ell}C\subset \PP^2.\] 
\end{definition}

This is, of course, also the intersection of any two distinct members of the pencil. Suppose first that $\ell$ contains a smooth conic, say $C_0$; in this case, if $C_1$ is any other member of $\ell$, then the base locus $B_\ell=C_0\cap C_1$ is a 0-dimensional subscheme of $\PP^4$ of length 4. The underlying reduced subscheme therefore is a union of at most 4 points. Up to $\PGL_3$-equivalence, there are exactly five such subschemes, and they are determined entirely by the multiplicities of the points, which must form a partition of 4, and hence be given by one of the following tuples: \[(1,1,1,1), (2,1,1), (2, 2), (3,1), (4).\]

\begin{definition}
Given a pencil $\ell$ of conics containing a smooth conic, we define the \textit{type} of $\ell$ to be the partition of 4 corresponding to the multiplicites of the points in its base locus.
\end{definition}

 Clearly, all five of these types are possible (simply take any two conics intersecting in such a subscheme), and the type of base locus is a $\PGL_3$-invariant of the pencil. This is in fact the \textit{only} invariant of pencils containing a smooth member, and this gives us five orbits of pencils not contained in $S$.

\begin{theorem}[Algebraic Description I]
There are exactly five $\PGL_3$-orbits of pencils of conics containing at least one nonsingular member. These are given by:
\begin{enumerate}
    \item The orbit $\OO_1=\PGL_3\langle x^2+y^2, x^2+z^2\rangle$ with base locus of type $(1,1,1,1)$ (general pencil).
    \item The orbit $\OO_2=\PGL_3\langle x^2+y^2, xz\rangle$ with base locus of type $(2,1,1)$ (simply tangent pencil).
    \item The orbit $\OO_3=\PGL_3\langle x^2, yz\rangle$ with base locus of type $(2,2)$ (bitangent pencil).
    \item The orbit $\OO_4=\PGL_3\langle x^2+yz, xz\rangle$ with base locus of type $(3,1)$ (osculating pencil).
    \item The orbit $\OO_5=\PGL_3\langle x^2, y^2+xz\rangle$ with base locus of type $(4)$ (superosculating pencil).
\end{enumerate}
\end{theorem}

For a proof, see \cite{Berger} Proposition 16.5.1 or \cite{Samuel} \S 1.7. From this description of the base locus, this portion of the Hasse diagram is immediate, see Figure \ref{fig: first_part_of_hasse}.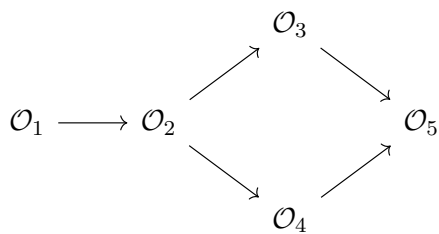
\begin{figure}[h]
    \centering
    \begin{tikzcd}
    &&\OO_3 \arrow{dr}\\
    \OO_1\arrow{r}&\OO_2\arrow{ur}\arrow{dr}&&\OO_5\\
    &&\OO_4\arrow{ur}
    \end{tikzcd}
    \caption{Hasse Diagram of the orbits containing at least one smooth conic.}
    \label{fig: first_part_of_hasse}
\end{figure}

We can also this interpret this theorem as a classification via the position of the line relative to the secant variety $S$ to the Veronese surface. Given a line $\ell\subset \PP^5$ not contained in $S$, let $C,C'\in \ell$ be two distinct elements and $Q, Q'$ be the corresponding symmetric $3\times 3$ matrices, so that the matrices $sQ+tQ'$ for $[s,t]\in \PP^1$ give a parametrization of $\ell$. The pullback of $\Delta$ to $\PP^1$ is exactly $\det(sQ+tQ')$, which is a cubic polynomial in $s$ and $t$.\footnote{This already proves the not-entirely-trivial fact that if a pencil contains at least one smooth member, then it has at most three singular elements; this corresponds to the geometric statement that a line $\ell$ not contained in $S$ meets it in at most $\deg S = 3$ distinct points.} The roots of this cubic polynomial then tell us how $\ell$ interesects $S$. For instance, for type $\OO_2$, we can take $C=\VV(x^2+y^2)$ and $C'=\VV(2xz)$ and so\[Q = \begin{bmatrix}1&0&0\\0&1&0\\
0&0&0\end{bmatrix} \text{ and }Q'= \begin{bmatrix}0&0&1\\0&0&0\\
1&0&0\end{bmatrix}.\]
Then \[sQ+tQ'=\begin{bmatrix}s&0&t\\0&s&0\\
t&0&0\end{bmatrix}\implies \det(sQ+tQ') = -st^2.\]
This tells us that this $\ell$ intersects $S$ in two smooth points, once with multiplicity 1 and another with multiplicity 2. Carrying out the analysis for all five types above, we get a geometric characterization of the orbits:

\begin{theorem}[Geometric Description I] The five $\PGL_3$-orbits of lines $\ell\subset \PP^5$ not contained in $S$ are characterized by their position relative to $S$, namely:
\begin{enumerate}
    \item $\OO_1$ is the orbit of lines intersecting $S$ transversely in three smooth points,
    \item $\OO_2$ is the orbit of lines intersecting $S$ in two smooth points, once with multiplicity 2,
    \item $\OO_3$ is the orbit of lines intersecting $S$ in one singular and one smooth point,
    \item $\OO_4$ is the orbit of lines intersecting $S$ in a unique smooth point with multiplicity 3, and
    \item $\OO_5$ is the orbit of lines intersecting $S$ in a unique singular point with multiplicity 3.
\end{enumerate}
\end{theorem}

Next, we look at the singular pencils (i.e. pencils consisting only of singular members). In this case, the base locus can be positive dimensional. 

\begin{theorem}[Algebraic Description II]
There are exactly 3 $\PGL_3$-orbits of singular pencils of conics. These are given by:
\begin{enumerate}
    \item The orbit $\OO_6 = \PGL_3\langle x^2, y^2\rangle$, the unique orbit of singular pencils of base loci consisting of a single point.
    \item The orbit $\OO_7 = \PGL_3\langle xy,xz\rangle$, of pencils with base locus consisting of a line $L$ and point $*\notin L$.
    \item The orbit $\OO_8 = \PGL_3\langle x^2, xy\rangle$, of pencils with base locus consisting of a line $L$ and an embedded point $*\in L$.
\end{enumerate}
\end{theorem}

A proof can be found in \cite{Samuel} \S 1.7. Doing the same analysis as before shows that we can understand this theorem as a classification via the position of the line relative to the Veronese surface $X$.
\begin{theorem}[Geometric Description II]
The three $\PGL_3$-orbits of lines $\ell\subset \PP^5$ contained in $S$ are characterized by their positions relative to $X$, namely:
\begin{enumerate}
    \item $\OO_6$ is the orbit of lines intersecting $X$ in two distinct points,
    \item $\OO_7$ is the orbit of lines not intersecting $X$ at all, and
    \item $\OO_8$ is the orbit of lines tangent to $X$.\footnote{A line cannot meet the Veronese surface in a positive dimensional scheme or a zero-dimensional scheme of length 3 or higher because the Veronese surface is cut out by quadrics and contains no lines (because of degree reasons).}
\end{enumerate}
\end{theorem}

In this case, the specialization results are not entirely obvious. The specialization $\OO_7\to \OO_8$ is easy to see from the base locus (algebraic) description, whereas the specialization $\OO_6\to \OO_8$ is easy to see from the geometric description. We will see below that $\OO_6$ and $\OO_7$ are both 4-dimensional and so cannot be comparable for the specialization partial order. Next, considering the family $\langle x^2, y^2+txz\rangle$ for $t\in \AA^1$ gives us the specialization $\OO_5\to \OO_6$. Similarly, considering the family $\langle xz, yz + tx^2\rangle$ for $t\in \AA^1$ gives us the specialization $\OO_4\to \OO_7$. What is also true but perhaps less obvious is that $\OO_5$ does not specialize to $\OO_7$; indeed, this is follows from the geometric description, since interesecting the Veronese surface $X$ is a closed condition in the Grassmannian. Putting all of this together, and combining it with the dimension computations of the orbits in the next section, we obtain the Hasse diagram, Fig. \ref{Hasse}, of orbit inclusions; c.f. \cite{HMS} \S 1.2 and \cite{AEI} \S 3.2. Finally, we also see that $\ol\OO_6$ and $\ol \OO_7$ are the two irreducible components of the Fano variety of $S$, as described in \cite[\S 2]{Lyu}. We will see in \S 3 that the Fano scheme of $S$ consists of $\ol\OO_6$ with multiplicity four and $\ol\OO_7$ with multiplicity one.

\newpage
\section{Calculations}

\subsection{Orbit $\OO_1$}

 The orbit $\OO_1$ is the orbit of pencils with base locus type $(1,1,1,1)$ and corresponds to lines $\ell\subset \PP^5$ meeting $S$ transversely in three smooth points. Therefore, $\OO_1$ is a nonempty open (and hence dense) orbit. Equivalently, if $U\subset \Hilb^4 \PP^2$ is the open subset of points in linear general position, then the map $U\to \GG(1,5)$ sending a base locus to the pencil of conics through those 4 points is an open embedding whose image is exactly $\OO_1$. The Plücker degree of its closure is the Plücker degree of $\GG(1,5)$ itself, namely 14, computed say using Lemma \ref{lemma: pluecker_degree_lines} as $\deg \Sigma_{0,0}$.

\subsection{Orbit $\OO_2$}

This is the orbit of pencils corresponding to lines simply tangent to $S$ at a smooth point.

\begin{proposition}
\label{proposition: orbit_O2}
The closure $\ol\OO_2$ is the variety $\ST_1(S)$ of lines tangent to $S$. In particular, $\OO_2$ is birational to a $\PP^3$-bundle over $S_\sm$, so $\dim \OO_2=7$.
\end{proposition}
\begin{proof}
Look at the incidence correspondence $\Phi_2$ defined in \S 2.1 above. The projection $\pi_2$ is dominant onto $S$. Given any $p\in S_\sm$, the fiber $\pi_2\inv(p)\subset\Phi_2$ is the set of lines tangent to $S$ at $p$, and these are exactly the lines through $p$ in $\TT_p S\cong \PP^4\subset\PP^5$.\footnote{In fact, suppose we have a smooth point $p=[LM]\in S_\sm$ corresponding to the product of two distinct lines $L,M\subset\PP^2$ and another point $q\in \PP^5\minus S$ with corresponding conic $C_q\subset \PP^2$. Then the line $\ell=\ol{p,q}$ is tangent to $S$ at $p$ iff $C_q\ni L\cap M$. In other words, the tangent hyperplane $\TT_p S$ is the hyperplane in $\PP^5$ of conics in the plane containing the point $L\cap M$.}It follows that $\Phi_2$ is the closure of projectivized geometric tangent bundle $\PP T S_{\sm}$ over $S_\sm$. Since $S$ is irreducible, it follows that $\Phi_2$ is irreducible of dimension 7, and birational to a $\PP^3$-bundle over $S_\sm$.

Now $\pi_1(\Phi_2)=:\ST_1(S)$ is the variety of tangent lines to $S$ (where we use the notation of \cite[\S 15.7]{Harris}; since $S$ is singular we need to be a little careful in how we define this variety). In here, $\OO_2$ is an open subset, since it is the complement of the locus of lines intersecting the Veronese surface or tangent to $S$ to order 3 or higher. It follows that $\ol \OO_2=\pi_1(\Phi_2)=\ST_1(S)$. Finally, the projection $\pi_1:\Phi_2\to \ol \OO_2$ is regular and birational, since a line simply tangent to the cubic hypersurface $S$ is tangent to it at a unique point.
\end{proof}

\begin{proposition}
\label{proposition: class_of_O2}
The class $[\ol\OO_2]=6\sigma_1$.
\end{proposition}

We give three proofs of this result, two here and one in \S 3.

\begin{proof}[Proof 1 of Proposition \ref{proposition: class_of_O2}.]
Write $[\ol\OO_2]=a\sigma_1$ for some $a\in \ZZ_{\geq 0}$. To figure $a$ out, we take the product on both sides with $\sigma_{4,3}\in A\GG(1,5)$. By Kleiman's transversality theorem, to calculate $a$, we're asking the question: given a general 2-plane $\Lambda\subset \PP^5$ and a general point $p\in \Lambda$, how many lines through $p$ and contained in $\Lambda$ are tangent to $S$? Since a general $2$-plane $\Lambda\subset \PP^5$ does not meet the two-dimensional $X$, Bertini's and Bezout's theorems tell us that for a general 2-plane $\Lambda$ the intersection $\Lambda\cap S$ will be a smooth plane cubic, and so given a general point $p\in \Lambda$, the number of lines through it tangent to this smooth cubic curve will be exactly the degree of the dual curve to a smooth cubic. It is well-known that the degree of the curve dual to a smooth plane curve of degree $d$ is $d(d-1)$; for $d=3$, this is $a=6$.
\end{proof}

\begin{proof}[Proof 2 of Proposition \ref{proposition: class_of_O2}]
To answer the same enumerative question differently, let $\Lambda$ be the span of the three points $p,p',p''$ corresponding to matrices $Q,Q',Q''$ respectively. The $\PP^1$ of lines in $\Lambda$ through $p$ then can be parametrized as $\langle Q, \lambda Q'+\mu Q''\rangle$ where $[\lambda,\mu]\in \PP^1$. Writing \[\det(sQ+t(\lambda Q'+\mu Q'')) = f_0(\lambda, \mu)s^3+f_1(\lambda,\mu) s^2 t+f_2(\lambda,\mu) st^2 +f_3(\lambda,\mu)t^3,\] where $f_{i}(\lambda,\mu)\in K[\lambda,\mu]$ is homogenous of degree $i$, we see that this cubic in $s$ and $t$ has a double root iff the discriminant \begin{align*} q_{6}(\lambda,\mu)&=\disc_{[s,t]}\det(sQ+t(\lambda Q'+\mu Q''))\\&=18f_0f_1f_2f_3 - 4f_1^3 f_3 +f_1^2f_2^2-4f_0f_2^2-27f_0^2f_3^2\end{align*} vanishes. This discriminant is a sextic polynomial and so generally has $a=6$ roots in $\PP^1$.
\end{proof}

\begin{corollary}
The Plücker degree $\deg \ol\OO_2=84$.
\end{corollary}
\begin{proof}
Taking $N=5, a=1, b=0$ in Lemma \ref{lemma: pluecker_degree_lines}, we see $\deg \Sigma_{1}=14$, from which it follows that $\deg \ol\OO_2 = 6\deg \Sigma_1 = 84$.
\end{proof}

\subsection{Orbit $\OO_3$}

This is the orbit of pencils that meet $S$ in two points, one of them being on $X$.

\begin{proposition}
The closure $\ol \OO_3$ is the variety $\SC_1(X)$ of lines incident to the Veronese surface. It is birational to a $\PP^4$ bundle over $\PP^2$, and so irreducible of dimension 6.
\end{proposition}
\begin{proof}
Consider the incidence correspondence \[\Psi:=\{(\ell,p):p\in \ell\cap X\}\subset \GG(1,5)\times X\]
and the projections $\pi_1$ and $\pi_2$ to $\GG(1,5)$ and $X$ respectively. This is simply the projectivized tautological quotient bundle restricted to $X$, i.e. $\Psi = \PP \mathcal{Q}|_X$, and so $\Psi$ is irreducible of dimension 6. On the other hand, $\pi_1(\Psi)=:\SC_1(X)$ and the projection $\pi_1$ is a regular birational map because a general line incident to $X$ is incident to it at a unique point\footnote{This follows from the existence of \textit{one} such line, which is easy to establish. An equivalent way of say thing is that the variety $\CS X\subset \GG(1,5)$ of secant lines to $X$, which has dimension 4, is a proper subvariety of $\SC_1(X)$.} It follows that $\SC_1(X)$ is also irreducible of dimension 6. Since $\OO_3$ is open in $\SC_1(X)$ (it is the complement in $\SC_1(X)$ of lines meeting $S$ at a single point to order 3 or more), it follows that $\ol\OO_3=\SC_1(X)$.
\end{proof}

\begin{proposition}
The class $[\ol \OO_3]=4\sigma_2$.
\end{proposition}

\begin{proof}
We write $[\ol\OO_3]=a\sigma_2+b\sigma_{1,1}$ for some $a,b\in \ZZ_{\geq 0}$. To compute $a$ and $b$, we take the product of $[\ol \OO_3]$ with $\sigma_{4,2}$ and $\sigma_{3,3}$ respectively. Taking the product with $\sigma_{4,2}$ amounts to asking: given a general 3-plane $\Pi\subset \PP^5$ and general point $p\in \PP$, how many lines containing a $p$ and contained in $\Pi$ intersect the Veronese surface? Since the Veronese surface has dimension 2 and degree 4, a general 3-plane $\Pi\subset\PP^5$ meets the Veronese surface in 4 distinct points in linear general position.\footnote{To show that these points are in linear general position, it suffices to observe that a general hyperplane section of the Veronese surface is a rational normal curve of degree 4.} Since a general point $p\in \Pi$ will not lie on the $\binom{4}{2}=6$ lines spanned by these points, it follows that there will be exactly $a=4$ lines through $p$ contained in $\Pi$ and meeting the Veronese surface.

Taking the product with $\sigma_{3,3}$ amounts to asking: given a general 2-plane $\Lambda\subset \PP^2$, how many lines intersecting the Veronese surface are contained in $\Lambda$? Well, a general 2-plane $\Lambda$ doesn't intersect the Veronese surface at all (for dimensional reasons), and so the answer is $b=0$.
\end{proof}

\begin{corollary}
The Plücker degree $\deg \ol\OO_3 = 36$.
\end{corollary}
\begin{proof}
Combine the previous result with $\deg \Sigma_2 = 9$ obtained from Lemma \ref{lemma: pluecker_degree_lines}.
\end{proof}

\subsection{Orbit $\OO_4$}

This is the orbit of pencils corresponding to the lines meeting $S$ in a single smooth point with multiplicity 3. 

\begin{proposition}
\label{proposition: orbit_O4}
The orbit $\OO_4$ is birational to a bundle over $S_\sm$ with each fiber consisting of two planes intersecting in a line and with nontrivial monodromy. In particular, $\ol\OO_4$ is irreducible of dimension 6.
\end{proposition}
\begin{proof}
Look at the incidence correspondence $\Phi_3$. On the one hand, again $\pi_2(\Phi_3) = S$. Given any $p=[LM]\in S_\sm$ corresponding to the product of the lines $L,M\subset \PP^2$ and another point $q\in \PP^5\minus S$ corresponding to the conic $C_q\subset \PP^2$, it is easy to see that the line $\ell=\ol{p,q}$ spanned by $p$ and $q$ intersects $S$ with multiplicity 3 at $p$ iff $C_q$ is tangent to either $L$ or $M$ at $L\cap M$. For a fixed $p$, the set of all such $q$ is the union of two 3-planes contained in $\TT_p S$ intersecting in the 2-plane of conics singular at $L\cap M$. In particular, given any $p\in S_\sm$, the fiber $\pi_2\inv(p)\subset \Phi_3$ is the projection of this locus from $p$, so the union of two 2-planes intersecting in a line.

 Further, this map $\pi_2:\Phi_3\to S$ has nontrivial monodromy (which is equivalent to saying that if $\Phi_3\to \hat S\to S$ is the Stein factorization of $\pi_2$, then the 2-sheeted cover $\hat S\to S$ is nontrivial). Indeed, this follows from the fact that after removing the $\PP^1$-bundle given by the intersection of the two planes that make up each fiber, the two irreducible components of the remaining fiber over a $p=[LM]$ correspond to tangency to $L$ or $M$; the lift of a loop on $S$ that corresponds to switching $L$ and $M$ then lifts to a path in the Stein factorization $\hat S$ that connects the two points on the fiber. In particular, it follows that $\Phi_3$ is irreducible of dimension 6.
 
  Since $\OO_4\subset\pi_1(\Phi_3)$ is an open subset (it is the complement of lines meeting the Veronese surface), it follows that $\ol\OO_4=\pi_1(\Phi_3)$. Finally, the map $\pi_{1}:\Phi_3\to \ol\OO_4$ is regular and birational because the fiber over any $\ell\in \OO_4$ is its unique point of intersection with $S$.
\end{proof}

\begin{proposition}
\label{proposition: class_of_O4}
The class $[\ol\OO_4]=6\sigma_2+9\sigma_{1,1}$.
\end{proposition}

Again, we give two proofs of this result, one here and one in $\S 3$.
\begin{proof}[Proof 1 of Proposition \ref{proposition: class_of_O4}]
Write $[\ol\OO_4]=a\sigma_2+b\sigma_{1,1}$ and take products with $\sigma_{4,2}$ and $\sigma_{3,3}$ as before.Taking the product with $\sigma_{4,2}$ amounts to asking: given a general 3-plane $\Pi\subset\PP^5$ and a general point $p\in \PP$, how many lines containing $p$ and contained in $\Pi$ are tangent to the cubic fourfold $S$ to order 3? If the intersection $S\cap\Pi$ were a smooth cubic surface, then the answer would be classical: the projection of smooth cubic surface in $\PP^3$ from a general point expresses it as a 3-sheeted cover of $\PP^2$ branched along a sextic with six cusps, and so there would be exactly $a=6$ such lines. The current problem is complicated by the fact $S\cap \Pi$ is not smooth; however, it is not too singular either (e.g. it is normal\footnote{It is normal by Serre's $R_1+S_2$ criterion because it is a hypersurface and regular in codimension 1.} and has finitely many isolated singularities), and the answer remains the same. Indeed, by Bertini's theorem, a general 3-plane $\Pi$ intersects $S$ in a cubic surface singular at 4 distinct points (namely the four points of $X\cap \Pi$), and there is only one such variety: the Cayley cubic, which is obtained by blowing up the plane along the 6 vertices of a complete quadrilateral (see \cite[pp. 640-646]{GH} , \cite{BW}, or \cite[\S 4.1.3]{Hunt}). General projections of cubic surfaces with ordinary double points have been studied as well (see \cite{FL}); since we'll give another proof of this result below, we'll content ourselves with quoting the results of \cite[\S 4.7]{FL} to conclude that $a=6$.

Finally, taking the product with $\sigma_{3,3}$ amounts to asking: given a general 2-plane $\Lambda\subset \PP^5$, how many lines contained in $\Lambda$ are triply tangent to $S$? By Bertini's and Bezout's theorems, a general $\Lambda$ intersects $S$ in a smooth cubic curve, and it is well-known that a smooth cubic has exactly $b=9$ flexes.
\end{proof}

\begin{corollary}
The Plücker degree $\deg \ol\OO_4 = 99$.
\end{corollary}
\begin{proof}
Combine the previous result with $\deg \Sigma_2 =9$ and $\deg \Sigma_{1,1}=5$.
\end{proof}

\subsection{Orbit $\OO_5$}

This the orbit of pencils that meet $S$ is a unique singular point (i.e. in a point of $X$) with multiplicity 3.

\begin{proposition}
The orbit $\OO_5$ is birational to a bundle over $\PP^2$ whose fiber is a rank 3 quadric threefold. In particular, $\ol\OO_5$ is irreducible of dimension 5.
\end{proposition}
\begin{proof}
Set up the incidence correspondance \[\Theta:=\{(\ell,p):p\in \ell \cap X\text{ and } m_p(S,\ell)\geq 3\}\subset \GG(1,5)\times X,\]
and let $\pi_1,\pi_2$ denotes projections to $\GG(1,5)$ and $X$ respectively. Given a point $p=[L^2]\in X$ corresponding to the double line $L^2\subset \PP^2$ and another point $q\in \PP^5\minus S$ corresponding to the conic $C_q\subset \PP^2$, it is easy to see that the line $\ell=\ol{p,q}$ joining $p$ and $q$ intersects $S$ with multiplicity 3 at $p$ iff $q$ lies in the tangent cone $\operatorname{TC}_p S$ to $S$ at $p$, which happens iff the conic $C_q$ is tangent to the line $L$ somewhere. The tangent cone $\operatorname{TC}_p S$ is a rank 3 quadric in $\PP^5$ containing $p$ in its vertex, and so for a fixed $p$, in the $\PP^4$ given by the projection of $\PP^5$ from $p$, this condition on $q$ defines a rank 3 quadric threefold $\pi_2\inv(p)\subset \Theta$. It follows that $\Theta$ is irreducible of dimension 5. Since $\OO_5$ is a nonempty open subset of $\pi_2(\Theta)$ (namely the complement in $\pi_2\Theta$ of the Fano variety $F_1(S)\subset \GG(1,5)$), it follows that $\ol\OO_5=\pi_2(\Theta)$. Finally, the map $\pi_2:\Theta\to \ol\OO_5$ is regular and birational (because the fiber over any point in $\OO_5$ is a single point).

\end{proof}

\begin{proposition}
The class $[\ol \OO_5]=4\sigma_{3}+8\sigma_{2,1}$.
\end{proposition}
\begin{proof}
Let $[\ol \OO_5]=a\sigma_3+b\sigma_{2,1}$. To figure $a$ out, we take the product with $\sigma_{4,1}$ and this amounts to asking: given a general point $q\in \PP^5$ and a general hyperplane $H\subset \PP^5$ containing $q$, how many lines through $q$ contained in $H$ meet $S$ in a unique point of $X$? We take $q\in \PP^5\minus S$. If $C_q\subset \PP^2$ is the corresponding smooth conic, then the line joining $q$ and a point $[L^2]\in X$ meets $S$ in $[L^2]$ only iff $L$ is tangent to $C_q$. In particular, the locus of all such $L$ is the dual conic $C_q^*\subset\PP^{2*}$, which under the Veronese map $\PP^{2*}\to X$ maps to a rational normal quartic in $\PP^5$. It follows that a general hyperplane $H$ containing $q$ interesects this curve in $a=4$ points. 
 
 To figure $b$ out, we take the product with $\sigma_{3,2}$, and this amounts to asking: given a general line $\ell\subset \PP^5$ and a general $3$-plane $\Pi\subset \PP^5$ containing $\ell$, how many lines incident to $\ell$ and contained in $\Pi$ meet $S$ in a unique point of $X$? By Bertini's theorem, a general 3-plane section $\Pi\cap X$ of $X$ consists of 4 points in linear general position, say $p_i=[L_i^2]$ for $i=1,2,3,4$, where the $L_i\subset \PP^2$ are four lines in linear general position. For each such point $p_i$, the tangent cone $\operatorname{TC}_{p_i}S$ intersects the 3-plane $\Pi$ in a rank 3 quadric surface, i.e. in a cone over a plane conic. No two of these quadric surfaces coincide, and so the intersection of any of two of them has positive codimension in either.\footnote{In fact, the intersection $\operatorname{TC}_{p_i}S\cap \operatorname{TC}_{p_j}S$ for $i\neq j$ is the locus of conics that are tangent to both $L_i$ and $L_j$. This locus has two components, namely the locus of conics through $L_i\cap L_j$ and another component. The first of these is a hyperplane along which $\operatorname{TC}_{p_i}S$ and $\operatorname{TC}_{p_j}S$ meet with multiplicity two, and whose intersection with $\Pi$ is a line, as is easily checked in local coordinates. The second of these defines a quadric in $\PP^5$ for degree reasons, and this intersects $\Pi$ in a plane conic.} Therefore, in the Grassmannian $\GG(1,\Pi)\cong \GG(1,3)$, the set of lines meeting the four quadric surfaces in 8 distinct points is also open, so $b=8$.
\end{proof}

\begin{corollary}
The Plücker degree $\deg \ol\OO_5 = 56$.
\end{corollary}
\begin{proof}
Combine the previous result with $\deg \Sigma_3=4$ and $\deg \Sigma_{2,1}=5$.
\end{proof}

\subsection{Orbit $\OO_6$}
This is the orbit of pencils which correspond to lines intersecting $X$ in two distinct points. 

\begin{proposition}
The closure $\ol\OO_6$ is the variety $\CS X$ of lines secant to $X$. It is birational to the Hilbert scheme $\Hilb^2\PP^{2*}$ of pairs of lines in the plane. In particular, $\ol\OO_6$ is irreducible of dimension 4.
\end{proposition}
\begin{proof}
It follows almost by definition that $\ol \OO_6=\CS X$, which is irreducible of dimension $\dim \CS X = 2\dim X =4$ (see \cite[Prop. 11.24]{Harris}). The map $\Sym^2 X\minus \Delta_X\to \CS X$ sending a pair of points $\{x,y\}$ with $x\neq y$ to the line spanned by them is an isomorphism onto $\OO_6$ and hence gives us a birational map between $\CS X$ and the Hilbert scheme $\Hilb^2 X$ of pairs on points on $X$, which is simply $\Hilb^2 \PP V^*$.
\end{proof}

Now we have:

\begin{proposition}
\label{proposition: class_of_O6}
The class $[\ol\OO_6]=3\sigma_{3,1}+6\sigma_{2,2}$.
\end{proposition}
\begin{proof}
Write $[\ol\OO_6]=a\sigma_4+b\sigma_{3,1}+c\sigma_{2,2}$. To figure $a$ out, we take the product with $\sigma_{4}$, which amounts to asking: given a general point $p\in \PP^5$, how many lines through $p$ are secant to $X$? Well, the Veronese surface is famously defective and the secant variety $S=S(X)$ is a proper subvariety of $\PP^5$, so the answer is simply $a=0$.

To figure $b$ out, we take the product with $\sigma_{3,1}$, and this amounts to asking: given a general line $\ell\subset\PP^5$ and a hyperplane $H$ containing $\ell$, how many secant lines to $X$ intersect $\ell$ and lie in $H$? By Bertini's and Bezout's theorems, for a general $H$ the intersection $H\cap X$ is a rational normal curve of degree 4 and the intersection $H\cap S$ is the secant variety of this rational normal curve. A general line $\ell\subset H$ meets $H\cap S$ in $\deg S = 3$ points. Generally, these points will not lie on the tangential surface to $H\cap X$. Each of these points lies on exactly one secant to the $H\cap X$, since otherwise we'd get 4 distinct coplanar points on it, which is impossible. In all, it follows that there are $b=3$ such secant lines intersecting $\ell$.

Finally, to figure $c$ out, we take the product with $\sigma_{2,2}$, and this amounts to asking: given a general 3-plane $\Pi\subset \PP^5$, how many secant lines to $X$ lie in $\Pi$? As observed several times above, a general 3-plane $\Pi$ intersects $X$ in 4 points in linear general position, so in general there are $c=\binom{4}{2}=6$ such secant lines.
\end{proof}

\begin{corollary}
The Plücker degree $\deg\ol\OO_6 = 21$.
\end{corollary}
\begin{proof}
Combine the previous result with $\deg \Sigma_{3,1}=3$ and $\deg\Sigma_{2,2}=2$.
\end{proof}

\subsection{Orbit $\OO_7$}

This is the orbit of pencils corresponding to lines $\ell\subset S$ not meeting $X$ at all, and is an open dense subset of the component of the Fano variety $F_1(S)$ other than $\CS X$.

\begin{proposition}
The variety $\ol\OO_7$ is isomorphic to the Segre fourfold $\PP^2\times \PP^{2*}$ and is in particular irreducible of dimension 4.
\end{proposition}
\begin{proof}
Consider the map $\PP^2\times\PP^{2*}\to \GG(1,5)$ sending a pair $(p,L)$ to the pencil of lines containing the fixed line $L$ and lines rotating around $p$. This is an injective map which is an embedding away from the locus $p\in L$. The image of this map is irreducible closed subset of $\GG(1,5)$ which contains $\OO_7$ as an open dense subset.
\end{proof}

\begin{proposition}
\label{proposition: class_of_O7}
The class $[\ol\OO_7]=6\sigma_{3,1}+3\sigma_{2,2}$.
\end{proposition}

Again, we give two proofs of this result, one here and one in \S 4.
\begin{proof}[First Proof of Proposition \ref{proposition: class_of_O7}.]
Write $[\ol\OO_7]=a\sigma_4+b\sigma_{3,1}+c\sigma_{2,2}$. To figure $a$ out, we take the product with $\sigma_4$, which amounts to asking: given a general point $p\in \PP^5$, how many lines through $p$ belong to $\ol\OO_7$? Since such lines are contained in $S$ and, as above, a general point $p$ doesn't lie in $S$ at all, the answer is $a=0$.

To figure $b$ out, we take the product with $\sigma_{3,1}$, which amounts to asking: given a general line $\ell\subset \PP^5$ and a hyperplane $H$ containing $\ell$, how many lines contained in the fourfold $S$ intersect $\ell$ and lie in $H$, other than the $3$ secants found above? Well, as before, the line $\ell$ meets $H\cap S$ in 3 points not lying on the tangent surface to $H\cap X$. Through each of these points, as it turns out and as we discuss below, there are exactly two lines contained in $S$ which are not secant to $X$, and so this gives us a total of $b=6$ such lines.

Finally, to calculate $c$, we take the product with $\sigma_{2,2}$, and this amounts to asking: given a general $3$-plane $\Pi\subset \PP^5$, how many lines in $\ol\OO_7$ lie in $\Pi$? This corresponds to lines in the Cayley cubic $\Pi\cap S$ \textit{other} than the 6 lines through its 4 ordinary double points, and this is well-known to be $c=3$.
\end{proof}

The remaining piece of the link is the following claim:

\begin{lemma}
\label{lemma: secant_threefold}
Let $\Gamma\subset \PP^4$ be the rational normal quartic. Given a general point\footnote{It suffices to take $p$ not lying on the tangent surface $\tan(\Gamma)$.} $p$ on the secant threefold $S_\Gamma:=S(\Gamma)$ to $\Gamma$, there are exactly three lines through $p$ contained in $S_\Gamma$, namely one line secant to $\Gamma$ and two other lines, from which $\Gamma$ is projected $2:1$ onto a smooth conic in the complementary plane.
\end{lemma}

To put this into context, we recall the well-known fact that given any point $p$ on a smooth cubic threefold $Y$, there are exactly 6 lines on $Y$ passing  through $p$. As this smooth cubic threefold specializes to the secant threefold $S_\Gamma$ (see the family explained in \cite[\S 2]{Collino}), one can study the specalization of this configuration of six lines. Segre showed in \cite{Segre} that the Fano variety $F_1(S_\Gamma)$ has two components, namely the variety $\CS(\Gamma)$ of secant lines to $\Gamma$, and another component of lines from which $\Gamma$ is projected $2:1$ onto a smooth conic in the complementary plane; these two components intersect in the variety $\ST(\Gamma)$ of tangent lines to $\Gamma$. It turns out that the first component has multiplicity 4 in the Fano scheme $F_1(S_\Gamma)$ (a proof is outlined in \cite[\S 2]{Collino}; we'll calculate this multiplicity below in \S 4 as well), and so counted properly this tells us through a given general point there should pass two lines of the second component. In stead of quoting this bit of theory, we give an alternative self-contained proof here.

\begin{proof}[Proof of Lemma \ref{lemma: secant_threefold}]
Think of $\PP^4=\PP H^0\OO_{\PP^1}(4)$ as the space of binary quartic forms.  The $j$-invariant of (the cross-ratio of) the four roots of the polynomial (see footnote 1) gives a rational function $j:\PP^4\dashrightarrow \PP^1$, whose fiber over a general point is a sextic hypersurface. There are two exceptional fibers corresponding to $j=0$ and $j=1728$, which are a triple quadric and a double cubic respectively. In fact, from classical invariant theory, we know that if we give $\PP^4$ the coordinates $a_0, \dots, a_4$ so that $[a_0,\dots, a_4]$ corresponds to the quartic polynomial $a_0x^4+a_1x^3y+\cdots+a_4y^4$, then these quadric and cubic surfaces are defined by the vanishing of the two invariants \begin{align*}I &= 12a_0a_4 - 3a_1a_3+a_2^2\text{ and }\\J&=72a_0a_2a_4-27a_0a_3^2-27a_1^2a_4+9a_1a_2 a_3-2a_2^3=\frac{1}{4}\det\begin{bmatrix}12a_0&3a_1&2a_2\\3a_1&2a_2&3a_3\\2a_2&3a_3&12a_4\end{bmatrix}\end{align*}
respectively. The latter, which will be more relevant to us, is (up to scaling the coordinates) the classical catalecticant or Hankel determinant of the binary quartic.

The space of all fourth powers of linear forms is a rational normal quartic $\Gamma\subset \PP^4$. A quartic has $j$-invariant 1728 iff by a $\PGL_2$ action its roots can be taken to $\{\pm 1, \pm i\}$ and this happens iff it is a linear combination of two fourth powers of linear forms, from which it follows that the cubic three-fold $\VV(J)$ is exactly the secant threefold $S_\Gamma$ to $\Gamma$. Now let $L_1, L_2$ be distinct linear forms on $\PP^1$; then for any $\lambda\in K^\times$, it is easy to see\footnote{By a $\PGL_2$ transformation, we can assume that $L_1=x$ and $L_2=y$, say.} from the explicit equation $J$ defining $S_\Gamma$ that there are exactly three lines in $S_\Gamma$ through the point $[L_1^4+\lambda L_2^4]\in S_\Gamma$, namely the secant line itself and the two lines joining $[L_1^4+\lambda L_2^4]$ to $[L_1L_2(L_1^2\pm \sqrt{-\lambda}L_2^2)]$. A general point on each of these two latter lines lies on a unique secant to $\Gamma$, so projection from such a line expresses $\Gamma$ as a two sheeted cover of a plane conic for degree reasons.
\end{proof}

\begin{corollary}
The Plücker degree $\deg\ol\OO_7 = 24$.
\end{corollary}
\begin{proof}
Combine the previous result with $\deg \Sigma_{3,1}=3$ and $\deg\Sigma_{2,2}=2$.
\end{proof}

\subsection{Orbit $\OO_8$}

This is the closed orbit of pencils corresponding to lines $\ell\subset S$ tangent to $X$.

\begin{proposition}
The orbit $\OO_8=\ST X$ is the variety of lines tangent to $X$. In particular, $\OO_8$ is a $\PP^1$ bundle over $X\cong \PP^2$ and so irreducible of dimension 3.
\end{proposition}
\begin{proof}
Clear.
\end{proof}

\begin{proposition}

The class $[\OO_8]=6\sigma_{4,1}+6\sigma_{3,2}$.
\end{proposition}
\begin{proof}
Write $[\OO_8]=a\sigma_{4,1}+b\sigma_{3,2}$. To figure $a$ out, we take the product with $\sigma_{3}$, which amounts to asking: given a general line $\ell\subset \PP^5$, how many lines intersecting $\ell$ are tangent to the Veronese surface? Well, a general $\ell$ intersects the tangent variety $\tan(X)=S(X)=S$ in $\deg S = 3$ smooth points, and each such point lies on exactly two tangents to the Veronese surface. Indeed, any point in $S\minus X$ lies on a unique 2-plane contained in $S$ which intersects $X$ in a plane conic. The existence of such a plane follows from the standard argument that proves the defectiveness of the Veronese surface, see \cite[Ch. 11]{Harris}. The uniqueness of such a plane follows from the fact that points (or more generally divisors) on rational normal curves are in linear general position, see Lemma 10.14 in \cite[\S 10.4]{EisenbudHarris}. Given this, the two tangents to $X$ through a point $p\in S\minus X$ are exactly the two tangents to this mentioned plane conic. It follows that there are $a=6$ such lines interesecting $\ell$ and tagent to $X$.

To figure $b$ out, we take the product with $\sigma_{2,1}$, which amounts to asking: given a general 2-plane $\Pi\subset \PP^5$ and a hyperplane $H\subset \PP^5$ containing $\Pi$, how many lines incident to $\Pi$ and contained $H$ are tangent to $X$? As above, the intersection $H\cap X$ is a rational normal curve of degree 4. The tangent variety $\tan(H\cap X)$ to such a curve is a surface of degree $2\cdot 4-2=6$, see \cite[Ex. 19.11]{Harris}. Therefore, a general 2-plane $\Pi\subset H$ intersects $\tan(H\cap X)$ in exactly 6 points, each of which lies on a unique tangent to the rational normal curve, again by the linear independence of divisors on rational normal curves. Therefore, there are exactly $b=6$ such tangent lines in general.
\end{proof}
\begin{corollary}
The Plücker degree $\deg \OO_8=18$.
\end{corollary}
\begin{proof}
Combine the previous result with $\deg \Sigma_{4,1}=1$ and $\deg \Sigma_{3,2}=2$.
\end{proof}
\newpage 
\section{Some Chern Class Computations}

\subsection{The Chow Ring Class of the Fano Scheme $F_1(S)$}

The first proof of Proposition \ref{proposition: class_of_O7} suggests that we should calculate the multiplicity of the irreducible components $\ol\OO_6$ and $\ol\OO_7$ of the Fano scheme $F_1(S)$ of lines on $S$. We claim:

\begin{proposition}
The component $\ol\OO_6$ of $F_1(S)$ has multiplicity 4 and the component $\ol\OO_7$ has multiplicity 1.
\end{proposition}
\begin{proof}
We know that every smooth cubic surface in $\PP^3$ contains exactly 27 lines, and these can be seen explicitly by considering $S$ as the blowup of the of the plane at a general configuration of 6 points\footnote{The precise condition is that the 6 points be in linear general position and that there is no conic passing through them. See \cite[\S 4.1]{GH} for more on this bit of classical theory.}; namely, these are the 6 exceptional divisors, the proper transforms of the $\binom{6}{2}=15$ lines joining these points, and the proper transforms of the six conics through all but one of the points. As this general configuration of 6 points specializes to the 6 vertices of a complete quadrilateral, it is easy to see how these lines specialize to the lines on the Cayley cubic and with what multiplicity; namely, we get the 6 exceptional divisors, each with multiplicity 4, and the three proper transforms of the diagonals of the complete quadrilateral.\footnote{See \cite[\S 4.6]{GH} for more details.} These same multiplicities must therefore also hold for the Fano scheme $F_1(S)$ of lines in $S$.
\end{proof}

Given this, we can write the class $[F_1(S)]=4[\ol\OO_6]+[\ol\OO_7]\in A\GG(1,5)$. Therefore, if we can figure our $[F_1(S)]$, then  we can use this to give another calculation of $[\ol\OO_7]$, since we gave a self-contained calculation of $[\ol\OO_6]$ in Proposition \ref{proposition: class_of_O6} above.

\begin{proposition}
\label{proposition: fano_scheme}
The class of the Fano scheme $[F_1(S)]=18\sigma_{3,1}+27\sigma_{2,2}$.
\end{proposition}
\begin{proof}[Proof 1 of Proposition \ref{proposition: fano_scheme}.]
We use \cite[Prop. 6.4]{EisenbudHarris}. Since $\dim F_1(S)=4$, its class in $A\GG(1,5)$ can be computed as the Chern class $c_4(\Sym^3 \CS^*)$, where $\CS\to \GG(1,5)$ is the tautological bundle. We know that $c(\CS^*)=1+\sigma_{1}+\sigma_{1,1}$. It follows exactly as in \cite[\S 6.2.1]{EisenbudHarris} from the splitting principle that \[c_4(\Sym^3\CS^*)=9\sigma_{1,1}(2\sigma_1^2+\sigma_{1,1})=18\sigma_{3,1}+27\sigma_{2,2}.\]
\end{proof}

Now we have:

\begin{proof}[Proof 2 of Proposition \ref{proposition: class_of_O7}]
We have \[[\ol\OO_7]=[F_1(S)]-4[\ol\OO_6]=18\sigma_{3,1}+27\sigma_{2,2}-4(3\sigma_{3,1}+6\sigma_{2,2})=6\sigma_{3,1}+3\sigma_{2,2}.\]
\end{proof}

\subsection{Chern Classes of Bundles of Relative Principal Parts}
\label{section: another_computation}

There is another construction that allows us to study the orbit closures mentioned above and their classes, namely the bundle of relative principal parts. This section will follow and use notation and terminology from \cite[\S 11.1]{EisenbudHarris}. Denote by $\Phi=\PP \CS$ the flag bundle that is the projectivized universal subbundle over $\GG(1,5)$, i.e. let \[\Phi:=\{(\ell,p):p\in\ell\}\subset \GG(1,5)\times\PP^5 ,\]

and let $\pi_1$ and $\pi_2$ denote its projections to $\GG(1,5)$ and $\PP^5$ respectively. Recall that the Chow ring $A(\Phi)$ of $\Phi$ is given by \[A(\Phi)= A\GG(1,5)[\zeta]/(\zeta^2-\sigma_{1}\zeta+\sigma_{1,1}),\]
where $\zeta=c_1(\OO_{\Phi}(1))=\pi_2^*\zeta_{\PP^5}$ is the pullback of a hyperplane class (i.e. a generator) in $A(\PP^5)$. Further, the pushforward map $\pi_{1,*}:A(\Phi)\to A\GG(1,5)$ is given by taking the coefficient of $\zeta$, i.e. \[\pi_{1,*}:A(\Phi)\to A\GG(1,5), \quad \alpha\zeta+\beta\mapsto \alpha.\]
For each $1\leq r\leq 4$, consider the bundle of relative principal parts  \[\SE^r:=\SP^{r-1}_{\Phi/\GG(1,5)}(\pi_2^*\OO_{\PP^5}(3))\]
over $\Phi$ of rank $r$ whose fiber over a point $(\ell,p)$ is \[\SE^r_{(\ell,p)}=H^0(\OO_\ell(3)/\fm_p^r(3)).\]

This admits a global section $\tau_{\Delta}\in H^0(\Phi, \SE^r)$ corresponding to the polynomial $\Delta$. The vanishing locus $\VV(\tau_\Delta)\subset \Phi$ is then exactly the locus $\Phi_r$ defined in \S 2. Therefore, it follows that if $\codim_\Phi \Phi_r=r$, then \[[\Phi_r]=c_r(\SE^r)\in A(\Phi)\]
is the $r^\text{th}$ Chern class of $\SE^r$. Finally, in the cases where this holds and $\pi_1|_{\Phi_r}:\Phi_r\to \GG(1,5)$ is birational onto its image, then $\pi_{1,*}[\Phi_r]=[\pi_1(\Phi_r)]$. Therefore, computing the Chern class of $\SE^r$ and taking its pushforward should tell us something about the class of the orbit closures. To compute the Chern class of $\SE^r$, we use the filtration of $\SE^r$ with successive quotients \[\pi_2^*\OO_{\PP^5}(3), \pi_2^*\OO_{\PP^5}(3)\otimes \Omega_{\Phi/\GG(1,5)},\cdots,\pi_2^*\OO_{\PP^5}(3)\otimes \Sym^{r-1}\Omega_{\Phi/\GG(1,5)},\] 
where $\Omega_{\Phi/\GG(1,5)}$ is the relative cotangent bundle of the map $\Phi\to \GG(1,5)$. Since \[c(\pi_2^*\OO_{\PP^5}(3))=1+3\zeta\text{ and }c(\Sym^{m}\Omega_{\Phi/\GG(1,5)})=1+m(\sigma_1-2\zeta),\]
it follows that \[c(\SE^r)=\prod_{m=0}^{r-1}(1+(3-2m)\zeta+m\sigma_1)\]
so that \[c_{r}(\SE^r)=\prod_{m=0}^{r-1}((3-2m)\zeta+m\sigma_1).\]

For $r=2,3,4$, these are easily seen to be \begin{align*}
    c_2(\SE^2) &= 6\sigma_1\zeta +3\sigma_{1,1},\\
    c_3(\SE^3) &= (6\sigma_2+9\sigma_{1,1})\zeta,\text{ and}\\
    c_4(\SE^4) &= 18\sigma_{3,1}+27\sigma_{2,2}
\end{align*}

From this we have:

\begin{proof}[Proof 3 of Proposition \ref{proposition: class_of_O2}]
We showed in Proposition \ref{proposition: orbit_O2} that $\dim \Phi_2 = 7$, so that $\codim_\Phi \Phi_2=\dim\Phi-\dim\Phi_2=9-7-2$. Further, the projection $\pi_1:\Phi_2\to \ol\OO_2$ is regular and birational, and so it follows that \[[\ol\OO_2]=[\pi_1(\Phi_2)]=\pi_{1,*}[\Phi_2]=\pi_{1,*}(6\sigma_1\zeta+3\sigma_{1,1})=6\sigma_1.\]
\end{proof}

\begin{proof}[Proof 2 of Proposition \ref{proposition: class_of_O4}]
We showed in Proposition \ref{proposition: class_of_O4} that $\dim \Phi_3=6$, so that $\codim_\Phi \Phi_3=3$. Further, the projection $\pi_1:\Phi_3\to \ol\OO_4$ is regular and birational, and so it follows that \[[\ol\OO_4]=[\pi_1(\Phi_3)]=\pi_{1,*}[\Phi_3]=\pi_{1,*}\left((6\sigma_2+9\sigma_{1,1})\zeta\right)=6\sigma_2+9\sigma_{1,1}.\]
\end{proof}

\begin{proof}[Proof 2 of Proposition \ref{proposition: fano_scheme}]
    The locus $\Phi_4 = \pi_2\inv F(S)$ (as a scheme!), so that it follows that $c_4(\SE^4) = [\Phi_4]=\pi_2^* [F(S)]$. Since $\pi_2^*$ is injective, the result follows.
\end{proof}
\newpage
\section{Dealing with Characterstic $p>2$}

In positive characteristic $p>2$, our results are still correct. Most of the above arguments remain valid; the only thing that changes is that we have to replace the use of Kleiman's transversality theorem by an explicit calculation of the generic transversality of the corresponding cycles. Since most of the cycles we are interested in incidence varieties, varieties of secant lines, varieties of tangent lines, or Fano varieties in the Grassmannian, we can use the explicit description of the tangent spaces to such cycles (as found in say \cite[Lect. 16]{Harris}) to check generic transversality. In this section, we check the result for one orbit closure for which such description is not as easy to write down (or locate a reference for), namely $\ol\OO_5$.

Recall that $\OO_5$ is the orbit of lines intersecting $S$ in a unique point of $X$. In an affine neighborhood of a line $\ell\in \OO_5$, it is straightforward to write down defining equations of the variety $\ol\OO_5$; to help in this process, you can take one point of such a line varying on $X$. Given this, it is very easy to figure out the tangent space $T_\ell \OO_5$ as a subspace of $T_\ell \GG(1,5)=\Hom(\ell, K^6/\ell)$. We see that if $p\in \ell\cap X$ is the unique point of intersection of $\ell$ with $X$, then if $H_\ell\subset K^6$ is the hyperplane of conics passing through the (underlying reduced subscheme, i.e. the point) of the base locus $B_\ell$, then \[T_\ell \OO_5 = \{\vphi\in \Hom(\ell, K^6/\ell): \vphi(p)\subseteq \TT_p X+\ell\text{ and } \vphi(\ell)\subseteq H_\ell\}.\]
This amounts to saying that as we vary the line $\ell$ along the tangent vector $\vphi$, to keep it incident to $X$ amounts to the condition $\vphi(p)\subseteq \TT_pX +\ell$, and to keep $\ell$ in the tangent cone to $S$ at its point of intersection with $X$ amounts to the additional condition $\vphi(\ell)\subseteq H_\ell$. From this, it follows that a general complementary Schubert cycle is generically transverse to $\OO_5$, which is the needed generic transversality result.
\newpage
\section{Suggestions for Further Work}

There are many natural ways to extend our results; we mention a few here.

\begin{enumerate}
    \item 
While $n=2, d=1$ is the last nontrivial case in which there are only finitely many orbits of the $\PGL_{n+1}$ action on $\GG(d, n(n+3)/2)$, for higher $n$ and $d$ there are still certain exceptional orbits, the classes in $A\GG(d, n(n+3)/2)$ of the closures of which can be calculated by methods similar to ours. For instance, one could look at the Chow ring classes of the 14 exceptional orbits of nets of plane conics ($n=d=2$) mentioned in \cite{AEI}.

\item The same also applies to linear systems of cubics or higher degree hypersurfaces. Even the case of plane cubics is already very interesting and nontrivial, c.f. \cite[\S 2.2]{EisenbudHarris} and the references there. In particular, see the book \cite{Greuel} by Greuel, Lossen and Shustin for more on what is known in the general case.

\item Finally, one can try to find the analogous results to ours in characteristic $p=2$, as Vainsencher does in \cite{Vainsencher} for the classical problem of the number of smooth conics tangent to 5 general conics (the answer being $51$ as opposed to $3264$ in characteristic $p\neq 2$).

\end{enumerate}

\newpage
\bibliographystyle{utphys}
\bibliography{main.bib}
\end{document}